\documentclass[11pt,a4paper]{amsart}
\usepackage{amssymb,amsmath,amsthm,graphicx}
\usepackage{tikz}
\usepackage{tikz-cd}
\usepackage{enumitem}
\usepackage{amsaddr}
\usepackage{etoolbox}
\usepackage{upgreek}\usetikzlibrary{matrix}
\usepackage[utf8]{inputenc}

\usepackage{thmtools,thm-restate}
\usepackage{float}
\usetikzlibrary{decorations.pathmorphing}
\usepackage{subcaption}
\usepackage{mathrsfs}
\usetikzlibrary{fit, shapes.geometric}
\usetikzlibrary{backgrounds}

\usepackage{hyperref}
\usepackage{cleveref}
\hypersetup{hypertexnames=false}
\usepackage{quiver}

\newcommand{\C}{\mathcal}
\newcommand{\s}{\mathsf}
\newcommand{\f}{\mathfrak}
\newcommand{\R}{\mathrm}

\newcommand{\ang}[1]{\langle#1\rangle}

\newcommand{\St}{\f{St}(\Lambda)}

\newcommand{\rightword}{\s{\overrightarrow{\s W}(M)}}

\newcommand{\leftword}{\s{\overleftarrow{\s W}(M)}}
\newcommand{\Apm}{\s A^{\pm}}
\newcommand{\Adpm}{{(\s A')}^{\pm}}

\newcommand{\weakbrick}[1]{\s{Br}^{\R{weak}}(\s #1)}
\newcommand{\Ba}{\f{Ba}(\Lambda)}

\newcommand{\INF}[1]{\,^\infty{#1}^\infty}

\newcounter{thmcount}
\newtheorem{defn}{Definition}[section]

\newtheorem{lem}[defn]{Lemma}

\newtheorem{prop}[defn]{Proposition}
\newtheorem*{prop*}{Proposition}
\newtheorem{thm}[defn]{Theorem}

\newtheorem{cor}[defn]{Corollary}

\newtheorem*{claim*}{Claim}

\newtheorem{exmp}[defn]{Example}
\AtEndEnvironment{exmp}{\hfill$\Diamond$}
\newtheorem{rmk}[defn]{Remark}

\numberwithin{equation}{section}

\title{An automata-based test for bricks over string algebras}
\author{Amit Kuber and Annoy Sengupta}

\address{Department of Mathematics and Statistics, Indian Institute of Technology Kanpur, Kanpur \\ 208016, Uttar Pradesh, India}
\email{askuber@iitk.ac.in, sengupta.annoy44@gmail.com}

\date{}


\begin{document}

\begin{abstract}
Motivated by the recent work of Deaconu, Mousavand and Paquette on the connection between infinite string bricks for certain gentle algebras and Sturmian words, we develop a decorated version of a deterministic automaton, called a \emph{multi-entry inverse automaton} (\emph{MIA}, for short) that accepts pointed words. We then associate an MIA $\s M_{\Lambda\delta}$ over $\{0,1\}$ to a string algebra $\Lambda$, and show that strings over $\Lambda$ can be viewed as certain equivalence classes of the pointed words accepted by $\s M_{\Lambda\delta}$. By defining \emph{(weak) brick words} over this MIA, we show that a finite/infinite string module (resp. band module) is a brick if and only if every word in the associated equivalence class of pointed binary words is a brick word (resp. a weak brick word) over $\s M_{\Lambda\delta}$. The result of Deaconu et al. follows as an immediate consequence.
\end{abstract}

\maketitle

\section{Introduction}
Fix an algebraically closed field $\C K$. One of the fundamental problems in the study of the representation theory of a finite-dimensional algebra over $\C K$ is the classification of indecomposable modules. For the class of string algebras, originally introduced by Butler and Ringel \cite{butler1987auslander} following the work of Gelfand and Ponomarev \cite{gel1968indecomposable}, the classification of finite-dimensional indecomposable modules is complete: every finite-dimensional indecomposable module is either a string module or a band module. These modules are constructed via purely combinatorial data derived from the underlying quiver and its relations.

Recently, there has is a flurry of research regarding \emph{bricks} \cite{mousavand2025bricks,mousavand2025biserial,dequêne2023generalization,sengupta2024characterisation,kaveh_sturmian}, which are modules having the endomorphism algebra isomorphic to the base field. By definition, bricks are indecomposable. Owing to this rigidity, the classification of bricks is often more tractable than the classification of all indecomposable modules, making them a natural first step toward understanding the category of representations.

\begin{figure}[h]
    \centering
    \[\begin{tikzcd}
v_1 & v_2 \arrow[l, "b_1", bend left, shift left=2] \arrow[l, "a_1"', bend right, shift right=2] & v_3 \arrow[l, "b_2", bend left, shift left=2] \arrow[l, "a_2"', bend right, shift right=2] & \cdots \arrow[l, "a_3"', bend right, shift right=2] \arrow[l, "b_3", bend left, shift left=2] & v_{N-1} \arrow[l, "a_{N-2}"', bend right, shift right=2] \arrow[l, "b_{N-2}", bend left, shift left=2] & v_N \arrow[l, "b_{N-1}", bend left, shift left=2] \arrow[l, "a_{N-1}"', bend right, shift right=2]
\end{tikzcd}\]
    \caption{$\Lambda_N$ with $\rho=\{a_1b_2,a_2b_3,\cdots,a_{N-2}b_{N-1},b_1a_2,b_2a_3,\cdots,b_{N-2}a_{N-1}\}$}
    \label{fig:lambda_n1}
\end{figure}
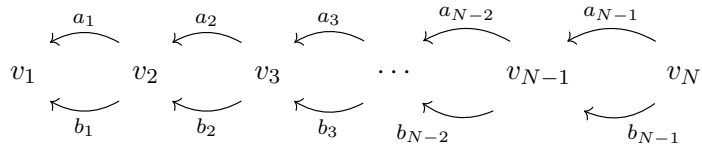

In the direction of the classification of bricks over string algebras, a surprising connection was established between some band bricks over a certain class $\Lambda_N$ of gentle algebras (\Cref{fig:lambda_n1}) and some words called \emph{perfectly clustering words} over a linearly ordered alphabet by Dequ\^ene et al. \cite{dequêne2023generalization}. This connection was generalised by the authors \cite{sengupta2024characterisation} to give a word-theoretic characterisation of all the band bricks over a large class of string algebras that includes gentle algebras. Very recently, the study of infinite-dimensional bricks has been advanced by Deaconu et al. \cite{kaveh_sturmian}, where they characterised infinite string bricks over the double Kronecker algebra $\Lambda_3$ (\Cref{fig:lambda_n1}) in terms of Sturmian words. These developments collectively highlight a modern trend where the ``brickness" of a module—whether finite or infinite—is captured through the lens of word combinatorics.

The goal of this paper is to generalise the main result of \cite{kaveh_sturmian}. It is important to understand Sturmian words before we discuss our contributions. The \emph{complexity} of an infinite word $\s w$ over the alphabet $\{\s a,\s b\}$ refers to the number of its distinct subwords of length $n$ for each $n\in\mathbb N$, where $\mathbb N$ denotes the set of all positive integers. The complexity of a periodic word is absolutely bounded. \emph{Sturmian words} \cite[Chapter~2]{lothaire1997combinatorics} are infinite aperiodic words over $\{\s a,\s b\}$ with minimal complexity among all aperiodic words. The complexity of a Sturmian word is exactly $n+1$ for each $n\in\mathbb N$.

Geometrically, Sturmian words can be visualised as the ``cutting sequence'' of an irrationally-sloped line on the integer grid, where intersection with a vertical (resp. horizontal) line counts as an occurrence of the letter $\s a$ (resp. $\s b$). A \emph{characteristic Sturmian word} is the unique ``cutting sequence'' obtained by translating this line to pass through the origin. Here are the characterisations of Sturmian words and characteristic Sturmian words that are useful to study bricks.

\begin{prop}\label{defn: Sturmian word}
\cite[Propositions~2.1.3, 2.1.22]{lothaire1997combinatorics} An infinite aperiodic word $\s w$ over $\{\s a,\s b\}$ is Sturmian if there is no finite word $\s w'$ such that both $\s a\s w'\s a$ and $\s b\s w'\s b$ are subwords of $\s w$.
A right-infinite Sturmian word $\s w$ is a characteristic Sturmian word if both $\s a\s w$ and $\s b\s w$ are Sturmian.
\end{prop}

Let $\s a:=b_1a_1^{-1}$, $\s b:=a_2^{-1}b_2$ be strings over $\Lambda_3$. Let $\f{S}(\s a,\s b)$ be the set of infinite strings over $\Lambda_3$ that could be written as infinite words over $\{\s a,\s b\}$. Now we are ready to state the main result of Deaconu et al.

\begin{thm}\cite[Theorem~4.10]{kaveh_sturmian}\label{thm: deaconu}
Let $\s w$ be an infinite word over $\{\s a,\s b\}$, and $\f{S}(\s w)\in\f{S}(\s a,\s b)$ be the corresponding infinite string.
\begin{enumerate}
    \item If $\s w$ is right-infinite, then $\s w$ is a characteristic Sturmian word if and only if the associated (direct sum) string module $M(\f{S}(\s w))$ is a brick.
    \item If $\s w$ is bi-infinite, then $\s w$ is a Sturmian word if and only if the associated (direct sum) string module $M(\f{S}(\s w))$ is a brick.
\end{enumerate}
\end{thm}

Replacing $\s a$ and $\s b$ by their ``parity words'' $\s d(\s a):=01$ and $\s d(\s b):=10$ respectively, the characterisation of a Sturmian word $\s w$ as given in \Cref{defn: Sturmian word} translates into the following statement: there does not exist a finite word $\s z$ over $\{0,1\}$ such that both $01\s z01$ and $10\s z10$ are subwords of $\s d(\s w)$. Note that the string corresponding to the word $\s z$ over $\{0,1\}$ is a cyclic string starting (and ending) at $v_2$, and that any infinite word over $\{0,1\}$ contains $1\s z 0$ (resp. $0\s z1$) as a subword if and only if it contains $01\s z01$ (resp. $10\s z10$) as a subword. Thus, in this restricted setting, a Sturmian word over $\{\s a,\s b\}$ could be characterised as a word $\s w$ over $\{0,1\}$ not containing a subword $\s z$ such that both $0\s z1$ and $1\s z0$ are subwords of $\s w$, provided we can capture the cyclicity of the string corresponding to $\s z$--we achieve this using automata.

An automaton provides an effective computational framework for recognising and generating sets of words over a given alphabet with prescribed restrictions. More formally, an \emph{automaton} over an alphabet $\s A$ is a machine that takes words over $\s A$ as input and decides whether the input word is ``accepted'' or not. A set $\s W$ of finite words over $\s A$ is said to be \emph{regular} if there exists an automaton $\s M$ such that the set of words accepted by $\s M$ is $\s W$. As a first interaction between the fields of representation theory of quivers and automata theory, Rees \cite{rees2008automata} constructed an automaton to show that the set of strings over a monomial algebra is regular. Continuing this thread of inquiry, Srivastava and the first author \cite{srivasta_automaton} automated the computation of the stable rank of a special biserial algebra.

In this paper, we build a decorated version of an automaton called \emph{a multi-entry inverse automaton (MIA)} (\Cref{defn: finite ptd automaton}) to furnish a word-theoretic characterisation of certain bricks over a string algebra. As the name suggests, an MIA has multiple initial states as opposed to a deterministic automaton having a single initial state. Therefore, this machine takes a word $\s w$ over $\s A$ and an initial state $v$ as input and then decides whether the ``pointed word'' $(v, \s w)$ is accepted or not. For a string algebra $\Lambda$ with underlying quiver $Q=(Q_0,Q_1)$, we give a prescription (\Cref{defn: MIA from string alg}) for building an MIA $\s M_\Lambda$ over the alphabet $Q_1\sqcup Q_1^{-1}$, and show (\Cref{prop: string are words}) that the set of finite (resp. infinite) strings over $\Lambda$ is in bijection with certain equivalence classes of finite (resp. infinite) pointed words accepted by $\s M_{\Lambda}$.

Traditionally, the algebraic property of ``brickness'' of an indecomposable module is captured by the existence of a common factor and image substring in the associated string  \cite{crawley1988tame,Krause1991MapsBT,crawley_boevey_infinite}. The most important contribution of this paper is the introduction of \emph{brick words} (\Cref{defn: brick word}) and \emph{weak brick words} (\Cref{defn: weak brick word}) in the framework of an MIA. Given an MIA $\s M$ over $\s A$, a certain surjective map $\varphi:\s A\to\s A'$ between alphabets called a \textit{local bijection} (\Cref{defn: local bijection}) allows us to define an MIA, denoted $M_\varphi$, over $\s A'$ so that there is a bijection (\Cref{thm: (weak) bricks are bijective}) between (weak) brick words over $\s M$ and $\s M_\varphi$. The restriction of the parity map $\delta:Q_1\sqcup Q_1^{-1}\to\{0,1\} $ with $\delta^{-1}(\{0\})=Q_1$ on $Q_1$ is a local bijection. The main results of the paper are the following.

\begin{thm}(\Cref{cor: string brick characterisation for string algebra automaton})
Given a string algebra $\Lambda$ and a finite/infinite aperiodic string $\f x$, the (direct sum) string module $M(\f x)$ is a brick if and only if every pointed word of the associated equivalence class of words over $\s M_{\Lambda\delta}$ is a brick word.
\end{thm}
\Cref{thm: deaconu} could be obtained as a consequence of the above (see \Cref{cor: DMP}).
\begin{thm}(\Cref{cor: band brick characterisation for string algebra automaton})
Given a string algebra $\Lambda$, a band $\f b$, $n\in\mathbb N$, and $\Lambda\in\C K^{\times}$, the band module $B(\f b,n,\lambda)$ is a brick if and only if $n=1$ and every pointed word of the associated equivalence class of words over $\s M_{\Lambda\delta}$ is a weak brick word.
\end{thm}


We outline a construction of a presentation of the string algebra $\Lambda$ from the MIA ${\s M_{\Lambda\delta}}$ in \S~\ref{sec: recover string alg}. 



\section{Preliminaries on string algebras}
In this section, we set up notation and outline some background for string algebras relevant to this paper. The presentation aligns with that of \cite{sengupta2024characterisation} with necessary minor variations. We begin by defining some combinatorial objects known as (finite and infinite) strings and bands consisting of words built using letters arising out of arrows in the quiver. These combinatorial objects are the key to constructing indecomposable modules known as string modules and band modules. These two classes of indecomposable modules exhaust all the finite-dimensional indecomposable modules over a string algebra; a proof, essentially due to \cite{gel1968indecomposable}, can be found in \cite{butler1987auslander}.

\begin{defn}\label{defn:string alg}
A \emph{string algebra} is a monomial algebra $\C KQ/\langle\rho\rangle$, where $Q=(Q_0,Q_1,s,t)$ is a finite quiver satisfying the following properties.
\begin{enumerate}[label=(\Roman*)]
    \item\label{indegoutdeg} For every $v\in Q_0$, there are at most two arrows $\alpha,\beta\in Q_1$ such that $s(\alpha)=s(\beta)=v$, and at most two arrows $\alpha',\beta'\in Q_1$ such that $t(\alpha')=t(\beta')=v$.
    \item\label{at most one arrow} For every $\alpha\in Q_1$, there is at most one arrow $\gamma\in Q_1$ such that $\alpha\gamma\notin\rho$ and at most one arrow $\gamma'$ such that $\gamma'\alpha\notin\rho$.
    \item\label{admissible} There exists $M\geq 0$ such that $|p|\leq M$ for every path $p$ having no subpath in $\rho$.
\end{enumerate}
If all the relations in $\rho$ are of length 2 and Condition \ref{at most one arrow} is replaced by the following conditions, then $\C KQ/\langle\rho\rangle$ is said to be a \emph{gentle algebra}.
\begin{enumerate}[label=(II\alph*)]
    \item For every $\alpha\in Q_1$, there is at most one arrow $\beta\in Q_1$ such that $\alpha\beta\notin\rho$ and at most one arrow $\beta'\in Q_1$ such that $\alpha\beta'\in\rho$.
    \item For every $\alpha\in Q_1$, there is at most one arrow $\gamma\in Q_1$ such that $\gamma\alpha\notin\rho$ and at most one arrow $\gamma'\in Q_1$ such that $\gamma'\alpha\in\rho$.
\end{enumerate}
\end{defn}
If $\Lambda=\C KQ/\langle\rho\rangle$ is a string algebra then we also say that $(Q,\rho)$ is a presentation of a string algebra $\Lambda$ and we refer $Q$ as the underlying quiver of $\Lambda$.

For each string algebra, it is useful to choose and fix maps $\varsigma,\varepsilon:Q_1\to\{-1,1\}$ with the following properties.
\begin{enumerate}[label=(\alph*)]
    \item If $\alpha,\beta\in Q_1$ such that $\alpha\neq\beta$ with $s(\alpha)=s(\beta)$, then $\varsigma(\alpha)=-\varsigma(\beta)$.
    
    \item If $\alpha,\beta\in Q_1$ such that $\alpha\neq\beta$ with $t(\alpha)=t(\beta)$, then $\varepsilon(\alpha)=-\varepsilon(\beta)$.
    
    \item If $\alpha,\beta\in Q_1$ such that $t(\alpha)=s(\beta)$ and $\alpha\beta\notin\rho$, then $\varepsilon(\alpha)=-\varsigma(\beta)$.
\end{enumerate}
For each $\alpha\in Q_1$, let $\alpha^{-1}$ be the formal inverse of $\alpha$ such that ${(\alpha^{-1})}^{-1}=\alpha$. We define $s(\alpha^{-1}):=t(\alpha)$ and $t(\alpha^{-1}):=s(\alpha)$. Define $Q_1^{-1}:=\{\alpha^{-1}:\alpha\in Q_1\}$. The elements of $Q_1\sqcup Q_1^{-1}$ will be called \emph{syllables}.

\begin{defn}\label{defn: string}
A finite sequence of syllables $\alpha_1\cdots\alpha_k$ is called a \emph{(finite) string} of length $k$ if the following hold.
\begin{itemize}
    \item For each $i\in\{1,\cdots,k-1\}$, $t(\alpha_i)=s(\alpha_{i+1})$.

    \item For each $i\in\{1,\cdots,k-1\}$, $\alpha_i\neq\alpha_{i+1}^{-1}$.

    \item For all $1\leq i\leq i+j\leq k$, $\alpha_i\cdots\alpha_{i+j}\notin\rho$ and $\alpha_{i+j}^{-1}\cdots\alpha_i^{-1}\notin\rho$.
\end{itemize}
In addition, for each $v\in Q_0$, there are zero-length strings $1_{(v,1)}$ and $1_{(v,-1)}$.
\end{defn}
For a string $\f x:=\alpha_1\cdots\alpha_k$, we extend the notions of inverse, starting and ending vertex in the obvious ways: $\f x^{-1}:=\alpha_k^{-1}\cdots\alpha_1^{-1}$, $s(\f x^{-1}):=t(\f x)$ and $t(\f x^{-1}):=s(\f x)$. We refer to $\alpha_1$ as the \emph{first syllable} of $\f x$ and $\alpha_k$ as the \emph{last syllable} of $\f x$. We also have $1_{(v,i)}^{-1}=1_{(v,-i)}$ and $s(1_{(v,i)})=t(1_{(v,i)})=v$ for any $i\in\{-1,1\}$. We denote by $\f{St}^{\R{fin}}(\Lambda)$ the set of all (finite) strings for a string algebra $\Lambda$ and by $\f{St}^{\R{fin}}_{>0}(\Lambda)$ the set of all positive-length (finite) strings. The $\varsigma$ and $\varepsilon$ maps are extended to $\f{St}(\Lambda)$ in the following way. For $\alpha\in Q_1$, $\varsigma(\alpha^{-1}):=\varepsilon(\alpha)$ and $\varepsilon(\alpha^{-1}):=\varsigma(\alpha)$. If $\f u=\alpha_1\cdots\alpha_k\in\f{St}_{>0}(\Lambda)$ then $\varsigma(\f u):=\varsigma(\alpha_1)$ and $\varepsilon(\f u):=\varepsilon(\alpha_k)$. In addition, we have $\varsigma(1_{(v,i)}):=-i=:-\varepsilon(1_{(v,i)})$ for each $i\in\{-1,1\}$.

For strings $\f x$ and $\f y$, if $\f x\f y$ is defined then $\varsigma(\f y)=-\varepsilon(\f x)$. For a vertex $v\in Q_0$ and $i\in\{-1,1\}$, we say that the concatenation $\f x1_{(v,i)}$ is defined and is equal to $\f x$ if $t(\f x)=v$ and $\varepsilon(\f x)=i$. Similarly, we say that the concatenation $1_{(v,i)}\f x$ is defined and is equal to $\f x$ if $s(\f x)=v$ and $\varsigma(\f x)=-i$.

Similar to finite strings, \emph{infinite strings} can be defined using an infinite sequence of syllables.

\begin{defn}
A sequence $\cdots\alpha_2\alpha_1$ of syllables is called a \emph{left $\mathbb N$-string} if $\alpha_k\cdots\alpha_1$ is a string for every $k\in\mathbb N$. Denote the set of all left $\mathbb N$-strings by $\overleftarrow{\f{St}}(\Lambda)$. A sequence $\alpha_1\alpha_2\cdots$ of syllables is called a \emph{right $\mathbb N$-string} if $\alpha_1\alpha_2\cdots\alpha_k$ is a string for every $k\in\mathbb N$. Denote the set of all left $\mathbb N$-strings by $\overrightarrow{\f{St}}(\Lambda)$. A sequence $\cdots\alpha_1\alpha_0\alpha_{-1}\cdots$ of syllables is called a \emph{$\mathbb Z$-string} if $\alpha_i\cdots\alpha_j$ is a string for every pair of integers $i\leq j$. Denote the set of all left $\mathbb N$-strings by $\overleftrightarrow{\f{St}}(\Lambda)$. Denote the set of all strings by $\St$.
\end{defn}

\begin{defn}
For (possibly infinite) strings $\f x$ and $\f y$, we say that $\f x$ is a \emph{substring} of $\f y$, denoted $\f x\sqsubseteq\f y$, if $\f y=\f u\f x\f v$ for some (possibly infinite) strings $\f u$ and $\f v$. We say that $\f x$ is a \emph{left substring} of $\f y$, denoted $\f x\sqsubseteq_l\f y$, if $\f y=\f x\f u$ for a (possibly infinite) string $\f u$. Dually say that $\f x$ is a \emph{right substring} of $\f y$, denoted $\f x \sqsubseteq_r \f y$, if $\f y=\f u\f x$ for a (possibly infinite) string $\f u$. We say that $\f x$ is a proper (resp. left, right) substring of $\f y$ if $\f x\sqsubseteq \f y$ (resp. $\f x\sqsubseteq_l \f y$, $\f x\sqsubseteq_r \f y$) but $\f x\neq\f y$.
\end{defn}

\begin{defn}\label{defn:factor image}
Given a (possibly infinite) string $\f x$, a (possibly infinite) substring $\f u$ of $\f x$ is said to be a \emph{factor substring} if one of the following holds:
\begin{itemize}
    \item $\f u=\f x$;
    \item $\f u\alpha\sqsubseteq_l\f x$ for some $\alpha\in Q_1$;
    \item  $\beta^{-1}\f u\sqsubseteq_r\f x$ for some $\beta\in Q_1$;
    \item $\beta^{-1}\f u\alpha\sqsubseteq\f x$ for some $\alpha,\beta\in Q_1$.
\end{itemize}
Dually, a (possibly infinite) substring $\f u$ of a (possibly infinite) string $\f x$ is said to be an \emph{image substring} if one of the following holds:
\begin{itemize}
    \item $\f u=\f x$;
    \item $\f u\alpha^{-1}\sqsubseteq_l\f x$ for some $\alpha\in Q_1$;
    \item $\beta\f u\sqsubseteq_r\f x$ for some $\beta\in Q_1$;
    \item $\beta\f u\alpha^{-1}\sqsubseteq\f x$ for some $\alpha,\beta\in Q_1$.
\end{itemize}
\end{defn}

\begin{defn}
A right $\mathbb N$-string $\f x=\alpha_1\alpha_2\alpha_3\cdots$ is called \emph{periodic} if there exists $k\in\mathbb N$ such that $\alpha_n=\alpha_{n+k}$ for every $n\in\mathbb N$. A right $\mathbb N$-string $\f x=\alpha_1\alpha_2\alpha_3\cdots$ over $\s A$ is called \emph{almost periodic} if $\f x$ has a periodic right substring. A right-infinite word $\f x=\alpha_1\alpha_2\alpha_3\cdots$ over $\s A$ is called \emph{aperiodic} if $\f x$ is not almost periodic.

A left $\mathbb N$-string $\f x=\cdots\alpha_3\alpha_2\alpha_1$ is called \emph{periodic} (resp. \emph{almost periodic}, \emph{aperiodic}) if $\f x^{-1}$ is periodic (resp. almost periodic, aperiodic).

A $\mathbb Z$-string $\f x=\cdots\alpha_{-1}\alpha_0\alpha_1\alpha_2\cdots$ is called \emph{periodic} if there exists $k\in\mathbb Z$ such that $\alpha_n=\alpha_{n+k}$ for every $n\in\mathbb Z$. The string $\f x$ is called \emph{right almost periodic} (resp. \emph{left almost periodic}) if the string $\alpha_1\alpha_2\alpha_3\cdots$ (resp. $\cdots\alpha_2\alpha_1\alpha_0$) is almost periodic. The word $\f x$ is called \emph{aperiodic} if it is neither left almost periodic nor right almost periodic.
\end{defn}

\begin{defn}
A string $\f b\in\s{St}^{\R{fin}}_{>0}(\Lambda)$ is called a \emph{band} if all of the following hold:
\begin{itemize}
    \item $\f b$ is cyclic i.e. $s(\f b)=t(\f b)$;
    \item $\f b$ is primitive i.e. $\f b\neq\f u^m$ for any $m\geq1$ and any string $\f u$;
    \item for every $m\geq1$, $\f b^m$ is a string;
    \item the first syllable of $\f b$ is inverse and the last syllable of $\f b$ is direct.
\end{itemize}
Denote by $\f{Ba}(\Lambda)$ the set of all bands for a string algebra $\Lambda$.
\end{defn}

For a finite string $\f x$, one can associate an indecomposable module known as a \emph{string module} $M(\f x)$ (see \cite[\S~1.9]{SchHam98} or \cite[\S~2.3.1]{LakThe16}). For an infinite string $\f x$, an infinite-dimensional direct-sum string module $M(\f x)$ can be defined similarly (see \cite[\S~2.3.3]{LakThe16}). We have $M(\f x)\cong M(\f y)$ if and only if $\f x=\f y$ or $\f x=\f y^{-1}$. For a band $\f b$, $l\geq1$ and $\lambda\in\C K^\times$, one can associate an indecomposable module known as a \emph{band module} $B(\f b,l,\lambda)$ (see \cite[\S~1.11]{SchHam98} or \cite[\S~2.3.2]{LakThe16}). We have $B(\f b,l,\lambda)\cong B(\f b',l',\lambda')$ if and only if $l=l'$, $\lambda=\lambda'$, and $\f b'$ is a cyclic permutation of either $\f b$ or $\f b^{-1}$. The following results give necessary and sufficient criteria for string modules and band modules to be bricks.

\begin{thm}\label{prop: string brick criteria}
\cite{crawley1989maps,crawley_boevey_infinite}
Let $\f x$ be a (possibly infinite) string. Then the string module $M(\f x)$ is a brick if and only if $\f x$ is aperiodic and there is no string $\f y$ that is a factor subtring of $\f x$ as well as an image substring of either $\f x$ or $\f x^{-1}$.
\end{thm}

\begin{thm}\label{prop: band brick criteria}
\cite{Krause1991MapsBT} Let $\f b$ be a band, $l\in\mathbb N$ and $\lambda\in\C K^\times$. Then the band module $B(\f b,l,\lambda)$ is a brick if and only if $l=1$ and there is no finite string $\f y$ that is a factor substring of $\INF{\f b}$ as well as an image substring of either $\INF{\f b}$ or $\INF{(\f b^{-1})}$.
\end{thm}

\section{Preliminaries on word combinatorics}
In this section, we mention some preliminaries and set up some notations for the combinatorics on words over an alphabet. A reader expecting an in-depth exploration of word combinatorics can refer to \cite{lothaire1997combinatorics}.

Let $\s A$ be an alphabet. Each element of $\s A$ is called a \emph{letter}.

\begin{defn}
A \emph{finite word} $\s w$ over $\s A$ is a finite sequence $\s b_1\s b_2\cdots\s b_n$ of letters over $\s A$. The finite words over $\s A$ are equipped with a binary operation called \emph{concatenation} defined by $(\s b_1\cdots\s b_n)(\s b'_1\cdots\s b'_m):=\s b_1\cdots\s b_n\s b'_1\cdots\s b'_m$. The length of a finite word is defined as $|\s b_1\cdots\s b_n|:=n$. We will also have a zero-length word denoted by $\epsilon$ which acts as an identity with respect to concatenation. We denote the set of all finite words over $\s A$ by $\s A^*$.
\end{defn}

\begin{defn}
A \emph{right-infinite word} over $\s A$ is an infinite sequence of letters $\s b_1\s b_2\cdots$. A \emph{left-infinite word} over $\s A$ is an infinite sequence of letters $\cdots\s b_2\s b_1$. The concatenation is extended to finite and these infinite words in the obvious way. A \emph{bi-infinite word} over $\s A$ is an infinite sequence of letters $\cdots\s b_{-1}\s b_0\s b_1\s b_2\cdots$. We denote the set of right-infinite words by $\overrightarrow{\s A}^*$, left-infinite words by $\overleftarrow{\s A}^*$, and bi-infinite words by $\overleftrightarrow{\s A}^*$. We set $\overline{\s A}^*:=\s A^*\sqcup\overrightarrow{\s A}^*\sqcup\overleftarrow{\s A}^*\sqcup\overleftrightarrow{\s A}^*$.
\end{defn}

\begin{defn}
Given words $\s w,\s w'\in\overline{\s A}^*$, we say that $\s w'$ is a subword of $\s w$, denoted by $\s w'\sqsubseteq\s w$ if $\s w=\s w''\s w'\s w'''$ for some word $\s w'',\s w'''\in\overline{\s A}^*$. We say that $\s w'$ is a left subword of $\s w$, denoted by $\s w'\sqsubseteq_l\s w$ if $\s w=\s w'\s w''$ for some word $\s w''\in\overline{\s A}^*$. We say that $\s w'$ is a right subword of $\s w$, denoted by $\s w'\sqsubseteq_r\s w$ if $\s w=\s w''\s w'$ for some word $\s w''\in\overline{\s A}^*$. We say that $\s w'$ is a proper (resp. left, right) subword of $\s w$ if $\s w'\sqsubseteq\s w$ (resp. $\s w'\sqsubseteq_l\s w$, $\s w'\sqsubseteq_r\s w$) but $\s w\neq\s w'$.
\end{defn}

Let $\s b^{-1}$ denote the formal inverse of a letter $\s b\in\s A$. Set $\s A^{-1}:=\{\s b^{-1}\mid\s b\in\s A\}$, $\Apm:=\s A\sqcup\s A^{-1}$, and $\s S(\s A):=\overline{\s A^\pm}^*$.

For a finite word or right-infinite word $\s w=\s b_1\s b_2\s b_3\cdots$, define $\s w^{-1}:=\cdots\s b_3^{-1}\s b_2^{-1}\s b_1^{-1}$. For a left-infinite word $\s w$, $\s w^{-1}$ is the unique left-infinite word such that ${(\s w^{-1})}^{-1}=\s w$. For a bi-infinite word $\s w=\s w^{\R{l}}\s w^{\R{r}}$, $\s w^{-1}:={(\s w^{\R{r}})}^{-1}{(\s w^{\R{l}})}^{-1}$.

\begin{defn}
A right-infinite word $\s w=\s b_1\s b_2\s b_3\cdots$ over $\s A$ is called \emph{periodic} if there exists $k\in\mathbb N$ such that $\s b_n=\s b_{n+k}$ for every $n\in\mathbb N$. A right-infinite word $\s w=\s b_1\s b_2\s b_3\cdots$ over $\s A$ is called \emph{almost periodic} if $\s w$ has a periodic right subword. A right-infinite word $\s w=\s b_1\s b_2\s b_3\cdots$ over $\s A$ is called \emph{aperiodic} if $\s w$ is not almost periodic.

A left-infinite word $\s w=\cdots\s b_3\s b_2\s b_1$ is called \emph{periodic} (resp. \emph{almost periodic}, \emph{aperiodic}) if $\s w^{-1}$ is periodic (resp. almost periodic, aperiodic).

A bi-infinite word $\s w=\cdots\s b_{-1}\s b_0\s b_1\s b_2\cdots$ is called \emph{periodic} if there exists $k\in\mathbb Z$ such that $\s b_n=\s b_{n+k}$ for every $n\in\mathbb Z$. The word $\s w$ is called \emph{right almost periodic} (resp. \emph{left almost periodic}) if the word $\s b_1\s b_2\s b_3\cdots$ (resp. $\cdots\s b_2\s b_1\s b_0$) is almost periodic. The word $\s w$ is called \emph{aperiodic} if it is neither left almost periodic nor right almost periodic.
\end{defn}

\section{Multi-entry inverse automaton}
In this section, we develop a decorated version of a finite automaton that is useful to construct strings over a string algebra.

\begin{defn}\label{defn: finite ptd automaton}
A \emph{multi-entry inverse automaton} (\emph{MIA}, for short) is a tuple $\s M:=(\s M_0,\s I\subseteq \s M_0, \s e:\s M_0\to \s I, (-)^{-1}:\s I\to\s I,\s A,\s t: \s M_0\times\Apm\rightharpoonup\s M_0)$, where $\s M_0$ is a finite set of \emph{states}, $\s I$ is a subset of $\s M_0$ containing \emph{initial states}, $\s A$ is a finite alphabet, and $\s t$ is a partial function governing state transition satisfying the following conditions:
\begin{enumerate}
    \item $(-)^{-1}$ is an involution without fixed points;
    \item $\s e(v)=v$ for all $v\in\s I$;
    \item if $\s t(v,\s b)$ is defined for $v\in\s M_0$ then $\s t(\s e(v),\s b)$ is defined and $\s e(\s t(v,\s b))=\s e(\s t(\s e(v),\s b))$.
\end{enumerate}
\end{defn}

Note that $\s M$ is a deterministic automaton with multiple initial states, and where all states are accepting states. An automaton having multiple initial states can be transformed into a traditional automaton by introducing a new unique initial state $\iota$ and defining zero-length transitions from $\iota$ to the old initial states.

Extend the transition function $\s t$ to $\widetilde{\s t}:\s M_0\times (\Apm)^*\rightharpoonup\s M_0$ as follows:
\begin{equation*}
\widetilde{\s t}(v,\s w):=
\begin{cases}
    v&\text{if }\s w=\epsilon,\\
    \s t(\widetilde{\s t}(v,\s w'),\s b)&\text{if }\s w=\s w'\s b\text{ for some }(\s w',\s b)\in(\Apm)^*\times\Apm.
\end{cases}
\end{equation*}

Since an MIA has multiple initial states unlike a traditional automaton, a word accepted by the automation has to include the data of the initial state. Keeping this in mind, we define below the formal notion of a \emph{word over an MIA}.

\begin{defn}
Let $\s M:=(\s M_0,\s I,\s e,(-)^{-1},\s A,\s t)$ be an MIA.
\begin{enumerate}
    \item A \emph{finite right word} over $\s M$ is a pair $(v,\s w)$, where $\s w\in(\Apm)^*$ and $v\in\s I$ such that $\widetilde{\s t}(v,\s w)$ is defined. Denote the set of all finite right words over $\s M$ by $\s{\overrightarrow{\s W}_{\R{fin}}(M)}$.

        \item An \emph{$\mathbb N$-word} over $\s M$ is a pair $(v,\s w)$, where $\s w$ is an infinite word over $\s A$ indexed by $\mathbb N$ and $v\in \s I$ such that $\widetilde{t}(v,\s w')$ is defined for each $\s w'\sqsubset_l\s w$. Denote the set of all $\mathbb N$-words over $\s M$ by $\s{\overrightarrow{\s W}_{\R{inf}}(M)}$. Set $\s{\overrightarrow{\s W}(M)}:=\s{\overrightarrow{\s W}_{\R{fin}}(M)}\cup \s{\overrightarrow{\s W}_{\R{inf}}(M)}$ and call each of its elements a \emph{right word}.

    \item A finite \emph{left word} over $\s M$ is a pair $(\s w,v)$ such that $(v^{-1},\s w^{-1})$ is a finite right word over $\s M$. Denote the set of all finite left words over $\s M$ by $\s{\overleftarrow{\s W}_{\R{fin}}(M)}$.
    
    \item An \emph{$\mathbb N^{\R{op}}$-word} over $\s M$ is a pair $(\s w,v)$ such that $(v^{-1},\s w^{-1})$ is an $\mathbb N$-word. Denote the set of all $\mathbb N^{\R{op}}$-words over $\s M$ by $\s{\overleftarrow{\s W}_{\R{inf}}(M)}$. Set $\s{\overleftarrow{\s W}(M)}:=\s{\overleftarrow{\s W}_{\R{fin}}(M)}\cup \s{\overleftarrow{\s W}_{\R{inf}}(M)}$ and call each of its elements a \emph{left word}.

    \item A \emph{word} over $\s M$ is a tuple $\s w:=(\s w^{\R{l}},v,\s w^{\R{r}})$, where $(\s w^{\R{l}},v)\in \leftword$, $(v,\s w^{\R{r}})\in\rightword$, $v\in\s I$, and $(\s w^{\R{l}}\s w',\s e(\widetilde{\s t}(v,\s w')))\in\leftword$ for each finite $\s w'\sqsubseteq_l\s w^{\R{r}}$. Denote the set of all words over $\s M$ by $\s W(\s M)$.
\end{enumerate}
\end{defn}

A word over an MIA can be intuitively visualised as a word over the alphabet with the additional data of an initial state of the MIA at a particular ``gap'' or ``basepoint'' of the word. With the help of the transition function $\s t$ and the function $\s e$, words can also be visualised as ``gap-labelled words''---words with additional data of an initial state at each gap. The next definition allows us to change the ``basepoint'' in a word over an automaton without changing the gap-labelled word.
\begin{defn}\label{defn: sim}
Given words $\s w_1:=(\s w^{\R{l}}_1,v_1,\s w^{\R{r}}_1)$ and $\s w_2:=(\s w^{\R{l}}_2,v_2,\s w^{\R{r}}_2)$ over $\s M$, we say that $\s w_1\sim\s w_2$ if there exists $\overline{\s w}\in \s S(\s A)$ such that
\begin{enumerate}
    \item $\s w^{\R{l}}_1\s w^{\R{r}}_1=\s w^{\R{l}}_2\s w^{\R{r}}_2$;
    \item $\s w_1^{\R{l}}=\s w^{\R{l}}_2\overline{\s w}$;
    \item $\s e(\widetilde{\s t}(v_2,\overline{\s w}))=v_1$.
\end{enumerate}
The closure of $\sim$ under symmetricity is again denoted by $\sim$.
\end{defn}

The proof of the next result is straightforward, and hence omitted.
\begin{prop}
The relation $\sim$ defined above is an equivalence relation.
\end{prop}

\begin{defn}\label{defn: subword}
Given words $\s w_1:=(\s w^{\R{l}}_1,v_1,\s w^{\R{r}}_1)$ and  $\s w_2:=(\s w^{\R{l}}_2,v_2,\s w^{\R{r}}_2)$ over $\s M$, we say that $\s w_1$ is a subword of $\s w_2$, denoted $\s w_1\sqsubseteq\s w_2$, if there exists $\s w_3:=(\s w^{\R{l}}_3,v_1,\s w^{\R{r}}_3)\sim \s w_2$ such that $\s w^{\R{l}}_1\sqsubseteq_r\s w^{\R{l}}_3$ and $\s w^{\R{r}}_1\sqsubseteq_l\s w^{\R{r}}_3$. We say that $\s w_1$ is a proper subword of $\s w_2$, denoted $\s w_1\sqsubset\s w_2$ if $\s w_1\sqsubseteq\s w_2$ and $\s w_1\not\sim\s w_2$.
\end{defn}

\begin{defn}\label{defn: factor and image subword}
Given words  $\s w_1:=(\s w^{\R{l}}_1,v_1,\s w^{\R{r}}_1)\sqsubseteq\s w_2:=(\s w^{\R{l}}_2,v_2,\s w^{\R{r}}_2)$ over $\s M$, let $\s w_3\sim\s w_2$ be a word as in the \Cref{defn: subword}. We call $\s w_1$ a \emph{factor subword} of $\s w_2$ if one of the following holds:
\begin{itemize}
    \item $\s w_1=\s w_3$;
    \item $\s w^{\R{l}}_1=\s w^{\R{l}}_3$ and $\s w^{\R{r}}_1\s b\sqsubset_l\s w^{\R{r}}_3$ for some $\s b\in\s A$;
    \item $\s w^{\R{r}}_1=\s w^{\R{r}}_3$ and $\s b^{-1}\s w^{\R{l}}_1\sqsubset_r\s w^{\R{l}}_3$ for some $\s b\in\s A$;
    \item$\s b^{-1}\s w^{\R{l}}_1\sqsubset_r\s w^{\R{l}}_3$ and $\s w^{\R{r}}_1\s b'\sqsubset_l\s w^{\R{r}}_3$ for some $\s b\in\s A$ and $\s b'\in \s A$.
\end{itemize}
We call $\s w_1$ an \emph{image subword} of $\s w_2$ if one of the following holds:
\begin{itemize}
    \item $\s w_3=\s w_4$;
    \item $\s w^{\R{l}}_3=\s w^{\R{l}}_3$ and $\s w^{\R{r}}_3\s b^{-1}\sqsubset_l\s w^{\R{r}}_4$ for some $\s b\in\s A$;
    \item $\s w^{\R{r}}_3=\s w^{\R{r}}_3$ and $\s b\s w^{\R{l}}_3\sqsubset_r\s w^{\R{l}}_4$ for some $\s b\in\s A$;
    \item$\s b\s w^{\R{l}}_3\sqsubset_r\s w^{\R{l}}_4$ and $\s w^{\R{r}}_3{\s b'}^{-1}\sqsubset_l\s w^{\R{r}}_4$ for some $\s b\in\s A$ and $\s b'\in \s A$.
\end{itemize}

\end{defn}

\begin{defn}\label{defn: brick word}
Say that $\s w:=(\s w^{\R{l}},v,\s w^{\R{r}})\in{\s W}(\s M)$ is a \emph{brick word} if $\s w^{\R{l}}\s w^{\R{r}}$ is aperiodic and there does not exist a word $\overline{\s w}=(\overline{\s w}^{\R{l}},\overline{v},\overline{\s w}^{\R{r}})$ which appears as a factor subword of $\s w$ and an image subword of $\s w$ or $\s w^{-1}$, either $\overline{\s w}\sqsubset\s w$ or $\overline{\s w}\sqsubset\s w$ or $\s w^{-1}$. Denote the set of all brick words by $\s{Br}(\s M)$.
\end{defn}

\begin{defn}\label{defn: weak brick word}
Say that $\s w:=(\s w^{\R{l}},v,\s w^{\R{r}})\in{\s W}(\s M)$ is a \emph{weak brick word} if there does not exist a finite word $\overline{\s w}=(\overline{\s w}^{\R{l}},\overline{v},\overline{\s w}^{\R{r}})$ which appears as a factor subword of $\s w$ and an image subword of $\s w$ or $\s w^{-1}$, either $\overline{\s w}\sqsubset\s w$ or $\overline{\s w}\sqsubset\s w$ or $\s w^{-1}$. Denote the set of all weak brick words by $\weakbrick{M}$.
\end{defn}

\begin{rmk}
It is easy to see that $\s{Br}(\s M)\subseteq\weakbrick{M}$. For finite words in $\s W(\s M)$, both notions coincide.
\end{rmk}

\begin{rmk}
Let $\s w_1$ and $\s w_2$ be two words over $\s M$ such that $\s w_1\sim\s w_2$. If $\s w_1$ is a brick word (resp. weak brick word), then $\s w_2$ is so too.
\end{rmk}

We now introduce a ``special'' map between alphabets which will allow us to construct a new automaton keeping its language(=set of all words) intact.

\begin{defn}\label{defn: local bijection}
Let $\s M:=(\s M_0,\s I,\s e,(-)^{-1},\s A,\s t)$ be an MIA and $\varphi:\s A\to\s A'$ be a surjective map of alphabets. Say that $\varphi$ is a \emph{local bijection for $\s M$} if for every $v\in\s M_0$, the restriction of the induced map $\varphi^\pm:\Apm\to\Adpm$ on the set $\{\s b\in\s A^{\pm}\mid\s t(v,\s b) \text{ is defined}\}$ is injective.
\end{defn}

Given a local bijection $\varphi:\s A\to \s A'$ for $\s M$, we can define an MIA $\s M_{\varphi}:=(\s M_0,\s I,\s e,(-)^{-1},\s A',\s t')$ by $\s t'(v,\s b'):=\s t(v,\s b)$ for $v\in\s M_0,\ \s b'\in\Adpm$ if there is $\s b\in\Apm$ such that $\varphi^\pm(\s b)=\s b'$ and $\s t(v,\s b)$ is defined--the uniqueness of such $\s b$ is guaranteed by \Cref{defn: local bijection}. 

The map $\varphi$ can be extended to a map $\s S\varphi:\s S(\s A)\to\s S(\s A')$ in the obvious way.

\begin{prop}\label{prop: words are in bijection after changing alphabet}
The map $\Phi:\s W(\s M)\to\s W(\s M_{\varphi})$ defined by $\Phi(\s w^{\R{l}},v,\s w^{\R{r}}):=({\s S\varphi}(\s w^{\R{l}}),v,{\s S\varphi}(\s w^{\R{r}}))$ is bijective.
\end{prop}

\begin{proof}
The definition of a local bijection guarantees that, for each $v\in\s M_0$, the restriction of $\s S\varphi$ to $\{\s w\mid (v,\s w)\in\overrightarrow{\s W}(\s M)\}$ is injective with image $\{\s w'\mid (v,\s w')\in\overrightarrow{\s W}(\s M_\varphi)\}$. Let $(\s S\varphi)^{-1}(v,-)$ denote the inverse of this restriction. Similarly, denote by $(\s S\varphi)^{-1}(-,v)$ the inverse of the restriction of $\s S\varphi$ to $\{\s w\mid (\s w,v)\in\overleftarrow{\s W}(\s M)\}$. The inverse of $\Phi$ is given by $((\s w')^{\R l},v,(\s w')^{\R r})\mapsto((\s S\varphi)^{-1}((\s w')^{\R l},v),v,(\s S\varphi)^{-1}(v,(\s w')^{\R r}))$.
\end{proof}





The above bijection $\Phi$ also preserves and reflects brick words and weak brick words.

\begin{cor}\label{thm: (weak) bricks are bijective}
Let $\s M:=(\s M_0,\s I,\s e,(-)^{-1},\s A,\s t)$ be an MIA and $\varphi:\s A\to\s A'$ be a local bijection for $\s M$. The bijection $\Phi:\s W(\s M)\to\s W(\s M_{\varphi})$ defined in \Cref{prop: words are in bijection after changing alphabet} restricts to bijections $\weakbrick{M}\cong\weakbrick{M_\varphi}$ and $\s{Br}(\s M)\cong\s{Br}(\s M_\varphi)$.
\end{cor}
\begin{proof}
The $\sim$-equivalence classes are preserved and reflected under $\Phi$. Hence the subwords are preserved and reflected under $\Phi$. The map $\s S\varphi$ sends letters in $\s A$ to $\s A'$ and $\s A'$ to $\s A'^{-1}$, and therefore $\Phi$ sends factor (resp. image) subwords to factor (resp. image) subwords. Hence, it is readily verified that $\Phi$ preserves and reflects bricks and weak bricks.
\end{proof}

\section{Strings as $\sim$-equivalence classes of words over an MIA}
Given a string algebra $\Lambda$, we now give a prescription to build an automaton $\s M_\Lambda$ that allows us to view strings for $\Lambda$ via words over $\s M_\Lambda$.

\begin{defn}\label{defn: MIA from string alg}
Let $\Lambda:=\C KQ/\ang{\rho}$ be a string algebra with $\St$ as its set of strings. The \emph{automaton associated with} $\Lambda$ is $\s M_{\Lambda}:=({\s M_{\Lambda}}_0,\s I_{\Lambda},\s e_{\Lambda},(-)^{-1},\s A_{\Lambda},\s t_\Lambda)$, where
\begin{align*}
    \s A_{\Lambda}&:=Q_1;\\
    \s I_\Lambda&:=\{\f x\in\St\mid|\f x|=0\};\\
    1_{(v,i)}^{-1}&:=1_{(v,-i)};\\
    {\s M_{\Lambda}}_0&:=\s I_\Lambda\cup Q_1\sqcup Q_1^{-1}\cup\{\f y\in\St\mid\f y\sqsubset_l\f x\text{, and } \{\f x,\f x^{-1}\}\cap\rho=1\};\text{ and}\\    
    \s t_\Lambda&:=\{(\f x,\s b,\f y)\mid \f x\s b\in\St,\  \f y\text{ is the maximal right substring of }\f x\s b\text{ such that }\f y\in {\s M_{\Lambda}}_0\}; \text{and}\\
    \s e_{\Lambda}(\f x)&:=1_{(t(\f x),\varepsilon(\f x))}.
\end{align*}
\end{defn}

\begin{prop}
The tuple $\s M_{\Lambda}$ is an MIA.
\end{prop}

The only non-trivial part in the proof of the above proposition is whether $\s t_\Lambda$ is a partial function. It follows from \Cref{rmk: outdegree < 2} below, which is followed from Condition $(II)$ in the definition of a string algebra.

\begin{rmk}\label{rmk: outdegree < 2}
Given a state $\f x\in {\s M_{\Lambda}}_0$, we have $$|\{\f y\in{\s M_{\Lambda}}_0\mid (\f x,\s b,\f y)\in \s t_\Lambda\text{ for some }\s b\in\s A_{\Lambda}\}|\leq1,\text{ and}$$
$$|\{\f y\in{\s M_{\Lambda}}_0\mid (\f x,\s b,\f y)\in \s t_\Lambda\text{ for some }\s b\in\s A_{\Lambda}^{-1}\}|\leq1.$$
\end{rmk}

\begin{rmk}
Each element of $\s I$ in $\s M_\Lambda$ is a source, i.e. for $v\in\s I$, there does not exist $v'\in{\s M_\Lambda}_0$ and $\s b\in\s A_{\Lambda}$ such that $\s t(v',\s b)=v$.
\end{rmk}

Let us recall from \cite[\S~1]{mousavand2023tau} the class of (generalised) barbell algebras. We present in Figure \ref{fig: barbell} a particular generalised barbell algebra and use it as a running example.

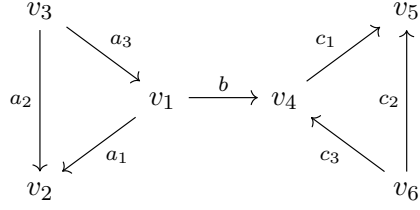
\begin{figure}[h]
    \centering
    \begin{tikzcd}
v_3 \arrow[rd, "a_3"] \arrow[dd, "a_2"'] &                                      &                       & v_5                                     \\
                                         & v_1 \arrow[r, "b"] \arrow[ld, "a_1"] & v_4 \arrow[ru, "c_1"] &                                         \\
v_2                                      &                                      &                       & v_6 \arrow[lu, "c_3"] \arrow[uu, "c_2"]
\end{tikzcd}
    \caption{$\Gamma$ with $\rho=\{a_3bc_1,a_3a_1,c_3c_1\}$}
    \label{fig: barbell}
\end{figure}

\begin{exmp}\label{exmp: barbell automaton}
Consider the string algebra $\Gamma$ in Figure \ref{fig: barbell}. We fix $\varsigma$ and $\varepsilon$ in the following way:
\begin{align*}
(\varsigma(a_1),\varepsilon(a_1))=(1,-1),\,&(\varsigma(a_2),\varepsilon(a_2))=(1,1),\,(\varsigma(a_3),\varepsilon(a_3))=(-1,1),\\
&(\varsigma(b),\varepsilon(b))=(-1,1),\\
(\varsigma(c_1),\varepsilon(c_1))=(-1,1),\,&(\varsigma(c_2),\varepsilon(c_2))=(1,-1),\,(\varsigma(c_3),\varepsilon(c_3))=(-1,-1).
\end{align*}

\Cref{fig: barbell automaton} shows the underlying automaton of the MIA $\s M_\Gamma$.
\end{exmp}
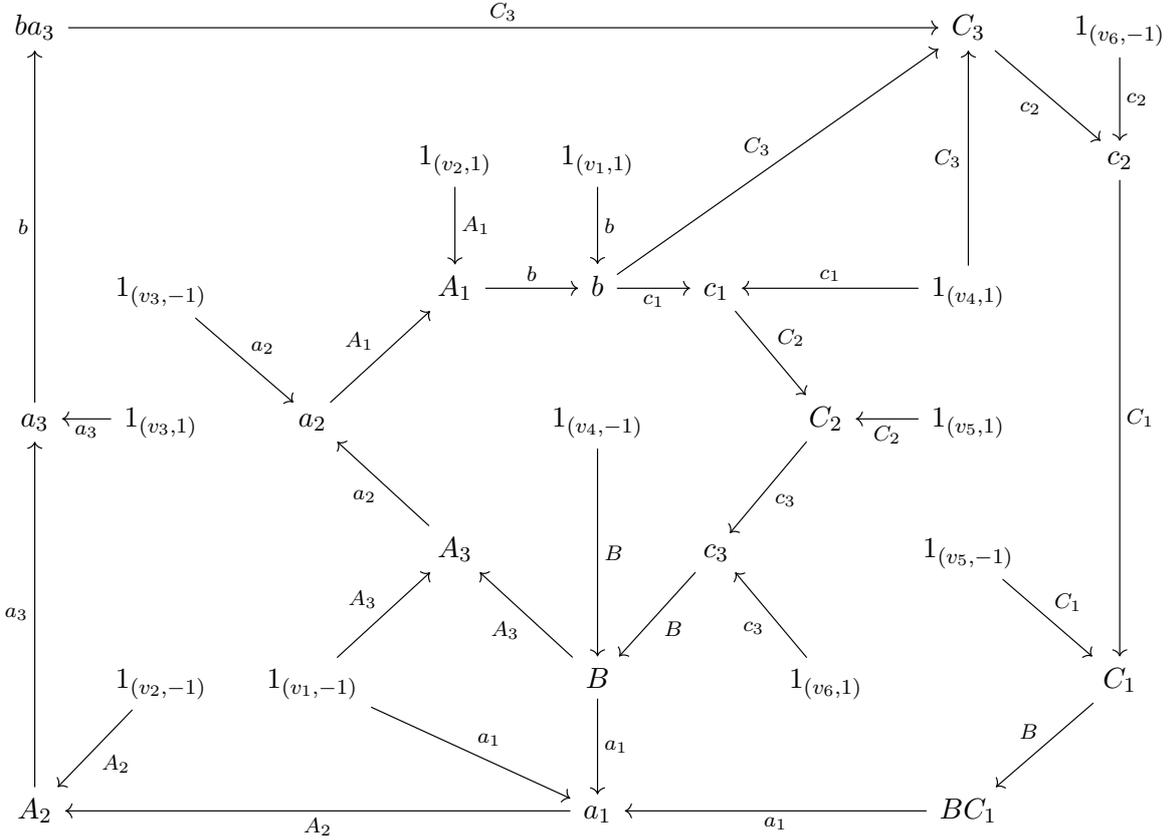
\begin{figure}[h]
   \[\begin{tikzcd}[ampersand replacement=\&,column sep=small, row sep=large]
	{ba_3} \&\&\&\&\&\&\& {C_3} \& {{1_{(v_6,-1)}}} \\
	\&\&\& {{1_{(v_2,1)}}} \& {{1_{(v_1,1)}}} \&\&\&\& {c_2} \\
	\& {{1_{(v_3,-1)}}} \&\& {A_1} \& b \& {c_1} \&\& {{1_{(v_4,1)}}} \\
	{a_3} \& {{1_{(v_3,1)}}} \& {a_2} \&\& {{1_{(v_4,-1)}}} \&\& {C_2} \& {1_{(v_5,1)}} \\
	\&\&\& {A_3} \&\& {c_3} \&\& {{1_{(v_5,-1)}}} \\
	\& {{1_{(v_2,-1)}}} \& {{1_{(v_1,-1)}}} \&\& B \&\& {{1_{(v_6,1)}}} \&\& {C_1} \\
	{A_2} \&\&\&\& {a_1} \&\&\& {BC_1}
	\arrow["C_3", from=1-1, to=1-8]
	\arrow["{c_2}"', from=1-8, to=2-9]
	\arrow["{c_2}", from=1-9, to=2-9]
	\arrow["{A_1}", from=2-4, to=3-4]
	\arrow["b", from=2-5, to=3-5]
	\arrow["{C_1}", from=2-9, to=6-9]
	\arrow["{a_2}", from=3-2, to=4-3]
	\arrow["b", from=3-4, to=3-5]
	\arrow["{C_3}", from=3-5, to=1-8]
	\arrow["{c_1}"', from=3-5, to=3-6]
	\arrow["{C_2}", from=3-6, to=4-7]
	\arrow["{C_3}", from=3-8, to=1-8]
	\arrow["{c_1}"', from=3-8, to=3-6]
	\arrow["b", from=4-1, to=1-1]
	\arrow["{a_3}", from=4-2, to=4-1]
	\arrow["{A_1}", from=4-3, to=3-4]
	\arrow["B", from=4-5, to=6-5]
	\arrow["{c_3}", from=4-7, to=5-6]
	\arrow["{C_2}", from=4-8, to=4-7]
	\arrow["{a_2}", from=5-4, to=4-3]
	\arrow["B", from=5-6, to=6-5]
	\arrow["{C_1}", from=5-8, to=6-9]
	\arrow["{A_2}", from=6-2, to=7-1]
	\arrow["{A_3}", from=6-3, to=5-4]
	\arrow["{a_1}", from=6-3, to=7-5]
	\arrow["{A_3}", from=6-5, to=5-4]
	\arrow["{a_1}", from=6-5, to=7-5]
	\arrow["{c_3}", from=6-7, to=5-6]
	\arrow["B"', from=6-9, to=7-8]
	\arrow["{a_3}", from=7-1, to=4-1]
	\arrow["{A_2}", from=7-5, to=7-1]
	\arrow["{a_1}", from=7-8, to=7-5]
\end{tikzcd}\]
    \caption{The underlying automaton of the MIA associated with the algebra in \Cref{exmp: barbell automaton}}
    \label{fig: barbell automaton}
\end{figure}

\begin{thm}\label{prop: string are words}
We have $\St\cong\s W(\s M_\Lambda)/_\sim$.
\end{thm}
\begin{proof}
Let $\f x\in\St$. Then $\f x=\f x1_{(t(\f x),\varepsilon(\f x))}$. Consider the map $\widetilde{\s W}:\St\to\s W(\s M_{\Lambda})/_\sim$ defined by $\widetilde{\s W}(\f x):=[(\f x,1_{(t(\f x),\varepsilon(\f x))},1_{(t(\f x),\varepsilon(\f x))})]_{\sim}$. If $\f x=\f x^{\R{l}}\f x^{\R{r}}$ for some $\f x^{\R{l}},\f x^{\R{r}}\in\St$, then $(\f x,1_{(t(\f x),\varepsilon(\f x))},1_{(t(\f x),\varepsilon(\f x))})\sim(\f x^{\R{l}},1_{(t(\f x^{\R{l}}),\varepsilon(\f x^{\R{l}}))},\f x^{\R{r}})$, and therefore the map $\widetilde{\s W}$ is well-defined.

On the other hand, consider the map $\f{St}:\s W(\s M_\Lambda)/_\sim\to\St$ defined by $\f{St}([\f x^{\R{l}},1_{(v,i)},\f x^{\R{r}}]_{\sim}):=\f x^{\R{l}}\f x^{\R{r}}$. This map is well-defined in view of Condition (1) of \Cref{defn: sim}. It can be easily checked that the maps $\widetilde{\s W}$ and $\f{St}$ are inverses of each other.
\end{proof}

The following lemmas are restatements of Theorems \ref{prop: string brick criteria} and \ref{prop: band brick criteria}.

\begin{lem}\label{prop: string brick characterisation for string algebra automaton}
Let $\f x\in\St$. Then $M(\f x)$ is a brick if and only if $\widetilde{\s W}(\f x)\subseteq\s{Br}(\s M_{\Lambda})$.
\end{lem}

\begin{lem}\label{prop: band brick characterisation for string algebra automaton}
Let $\f b\in\Ba$, $l\in\mathbb N$, $\lambda\in\C K^\times$. Then $B(\f b,l,\lambda)$ is a brick if and only if $l=1$ and $\widetilde{\s W}(\,^\infty\f b^\infty)\subseteq\s{Br}^{\R{weak}}(\s M_{\Lambda})$.
\end{lem}

\begin{proof}[Proofs of Lemmas \ref{prop: string brick characterisation for string algebra automaton} and \ref{prop: band brick characterisation for string algebra automaton}]
Under the map $\widetilde{\s W}$ defined in the proof of \Cref{prop: string are words}, substrings are sent to subwords. More precisely, if $\f x\sqsubseteq\f y$, then $\s w\sqsubseteq\s w'$, where $\s w\in\widetilde{\s W}(\f x)$ and $\s w'\in\widetilde{\s W}(\f y)$. It follows from the definition of factor and image substrings that $\f x$ is a factor (resp. image) substring of $\f y$ if and only if $\s w$ is a factor (resp. image) subword of $\s w'$. The rest of the argument follows from Theorems \ref{prop: string brick criteria} and \ref{prop: band brick criteria}.
\end{proof}


\Cref{rmk: outdegree < 2} guarantees that if $Q_1\neq\emptyset$ then the constant map $\delta:Q_1\to\{0\}$ is a local bijection for $\s M_\Lambda$ since $\Lambda$ is a string algebra. Set $\s A':=\{0\}$, and denote $0^{-1}$ by $1$ so that $(\s A')^\pm=\{0,1\}$ is the binary alphabet and $\delta^\pm:Q_1^\pm\to\{0,1\}$ is the parity map.

\begin{exmp}\label{exmp: barbell automaton delta}
\Cref{fig: barbell automaton delta} shows the underlying automaton of the MIA ${\s M_{\Gamma\delta}}$.
\end{exmp}

\begin{figure}[H]
   \[\begin{tikzcd}[ampersand replacement=\&,column sep=small, row sep=large]
	{ba_3} \&\&\&\&\&\&\& {C_3} \& {{1_{(v_6,-1)}}} \\
	\&\&\& {{1_{(v_2,1)}}} \& {{1_{(v_1,1)}}} \&\&\&\& {c_2} \\
	\& {{1_{(v_3,-1)}}} \&\& {A_1} \& b \& {c_1} \&\& {{1_{(v_4,1)}}} \\
	{a_3} \& {{1_{(v_3,1)}}} \& {a_2} \&\& {{1_{(v_4,-1)}}} \&\& {C_2} \& {1_{(v_5,1)}} \\
	\&\&\& {A_3} \&\& {c_3} \&\& {{1_{(v_5,-1)}}} \\
	\& {{1_{(v_2,-1)}}} \& {{1_{(v_1,-1)}}} \&\& B \&\& {{1_{(v_6,1)}}} \&\& {C_1} \\
	{A_2} \&\&\&\& {a_1} \&\&\& {BC_1}
	\arrow["1", from=1-1, to=1-8]
	\arrow["{0}"', from=1-8, to=2-9]
	\arrow["{0}", from=1-9, to=2-9]
	\arrow["{1}", from=2-4, to=3-4]
	\arrow["0", from=2-5, to=3-5]
	\arrow["{1}", from=2-9, to=6-9]
	\arrow["{0}", from=3-2, to=4-3]
	\arrow["0", from=3-4, to=3-5]
	\arrow["{1}", from=3-5, to=1-8]
	\arrow["{0}"', from=3-5, to=3-6]
	\arrow["{1}", from=3-6, to=4-7]
	\arrow["{1}", from=3-8, to=1-8]
	\arrow["{0}"', from=3-8, to=3-6]
	\arrow["0", from=4-1, to=1-1]
	\arrow["{0}", from=4-2, to=4-1]
	\arrow["{1}", from=4-3, to=3-4]
	\arrow["1", from=4-5, to=6-5]
	\arrow["{0}", from=4-7, to=5-6]
	\arrow["{1}", from=4-8, to=4-7]
	\arrow["{0}", from=5-4, to=4-3]
	\arrow["1", from=5-6, to=6-5]
	\arrow["{1}", from=5-8, to=6-9]
	\arrow["{1}", from=6-2, to=7-1]
	\arrow["{1}", from=6-3, to=5-4]
	\arrow["{0}", from=6-3, to=7-5]
	\arrow["{1}", from=6-5, to=5-4]
	\arrow["{0}", from=6-5, to=7-5]
	\arrow["{0}", from=6-7, to=5-6]
	\arrow["1"', from=6-9, to=7-8]
	\arrow["{0}", from=7-1, to=4-1]
	\arrow["{1}", from=7-5, to=7-1]
	\arrow["{0}", from=7-8, to=7-5]
\end{tikzcd}\]
    \caption{The underlying automaton of the MIA ${\s M_{\Gamma\delta}}$ used in \Cref{exmp: barbell automaton delta}}
    \label{fig: barbell automaton delta}
\end{figure}
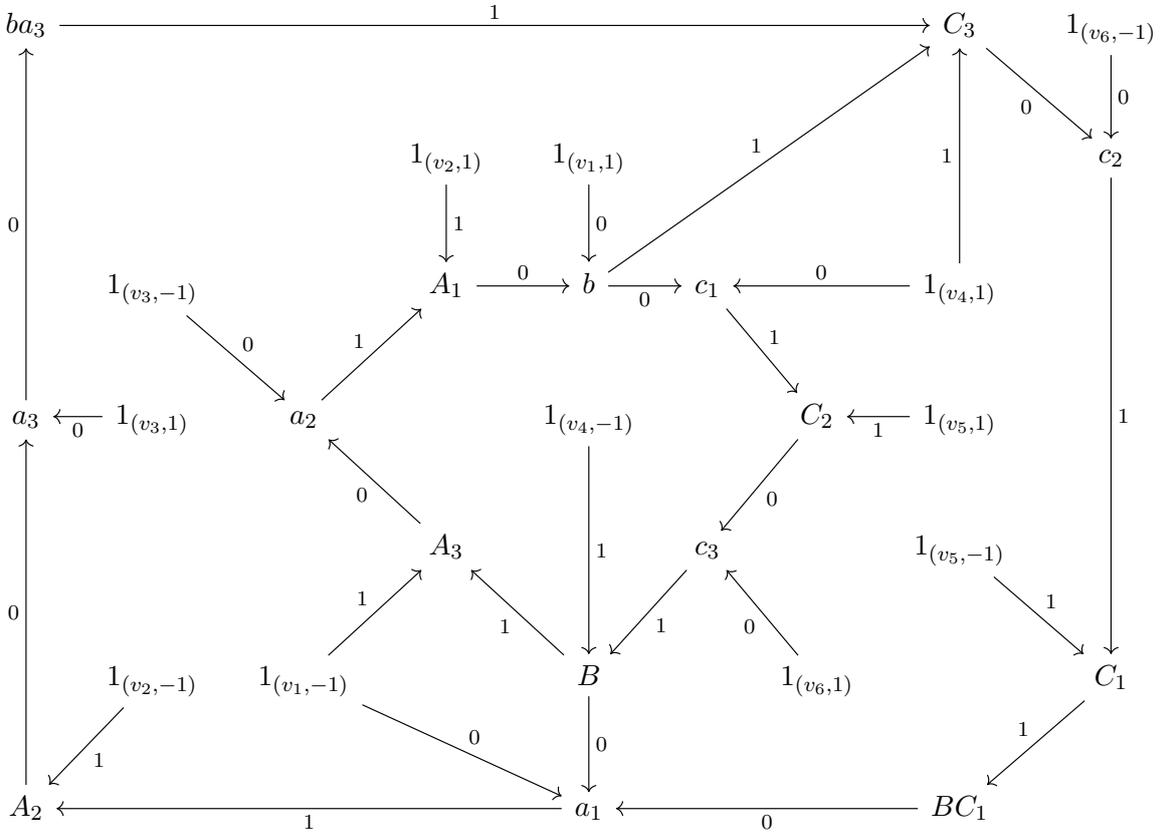

Applying \Cref{thm: (weak) bricks are bijective} to Lemmas \ref{prop: string brick characterisation for string algebra automaton} and \ref{prop: band brick characterisation for string algebra automaton}, we have the following characterisation of brick modules over string algebras in terms of binary (weak) brick words.

\begin{thm}\label{cor: string brick characterisation for string algebra automaton}
Let $\f x\in\St$. Then $M(\f x)$ is a brick if and only if $\Phi(\widetilde{\s W}(\f x))\subseteq\s{Br}(\s M_{\Lambda\delta})$.
\end{thm}

\begin{thm}\label{cor: band brick characterisation for string algebra automaton}
Let $\f b\in\Ba$, $l\in\mathbb N$, and $\lambda\in\C K^\times$. Then $B(\f b,l,\lambda)$ is a brick if and only if $l=1$ and $\Phi(\widetilde{\s W}(\,^\infty\f b^\infty))\subseteq\s{Br}^{\R{weak}}(\s M_{\Lambda\delta})$.
\end{thm}

\section{Connection with Sturmian words}\label{sec:connection}
Consider the double Kronecker algebra $\Lambda_3$ (\Cref{fig:lambda_n1}). Let $\f x\in\overleftrightarrow{\f{St}}(\Lambda_3)$ be an aperiodic string. With the assumption of aperiodicity, \Cref{cor: string brick characterisation for string algebra automaton} yields that the string module $M(\f x)$ is not a brick module if and only if $\Phi(\widetilde{\s W}(\f x))\not\subseteq\s{Br}^{\R{weak}}(\s M_{\Lambda_3\delta})$. In other words, there is $\s w'\in\{0,1\}^*$ such that both $0\s w'1$ and $1\s w'0$ are subwords of the underlying word $\s w(\f x)\in\overleftrightarrow{\{0,1\}}^*$ of $\Phi(\widetilde{\s W}(\f x))$. (An observation of the underlying quiver of $\Lambda_3$ yields a cyclic string $\f x'$ with $s(\f x')=t(\f x')=v_2$ such that $\s w(\f x')=\s w'$.) Given that $\f x$ is bi-infinite, this happens if and only if both $10\s w'10$ and $01\s w'01$ are subwords of $\s w(\f x)$.

Let $\s a:=b_1A_1,\s b:=A_2b_2\in\f{St}(\Lambda_3)$ so that $\s w(\f x)$ could be written as a (bi-infinite) word $\s z(\f x)$ over $\{\s a,\s b\}$. Using this new notation, the condition in the last line of the above paragraph translates to the existence of a finite word $\s y'$ over $\{\s a,\s b\}$ such that both $\s a\s y'\s a$ and $\s b\s y'\s b$ are subwords of $\s z(\f x)$. In view of \Cref{defn: Sturmian word}, the above discussion could be summarised as Part $(1)$ of the following corollary of \Cref{cor: string brick characterisation for string algebra automaton}--the second part admits a similar proof thanks to \Cref{defn: Sturmian word}.
\begin{cor}[A restatement of \Cref{thm: deaconu}]\label{cor: DMP}
Let $\f x\in{\f{St}}(\Lambda_3)$ be infinite and aperiodic.
\begin{enumerate}
    \item If $\f x$ is bi-infinite, then the module $M(\f{x})$ is a brick if and only if $\s z(\f x)$ is a Sturmian word over $\{\s a,\s b\}$.
    \item If $\s w$ is right-infinite and $s(\f x)=v_2$, then the module $M(\f{x})$ is a brick if and only if $\s z(\f x)$ is a characteristic Sturmian word over $\{\s a,\s b\}$.
\end{enumerate}
\end{cor}

\section{Recovering a string algebra from its associated MIA}\label{sec: recover string alg}
This short section is intended to give a prescription for recovering a string algebra $\Lambda=\C KQ/\ang{\rho}$ from $\s M_{\Lambda\delta}$.

\begin{defn}\label{defn: string alg from MIA}
Let $\s M:=(\s M_0,\s I,\s e,(-)^{-1},\{0\},\s t)$ be an MIA. We construct a quiver $Q_\s M$ whose vertex set is ${Q_{\s M}}_0:=\s I/\approx,\text{ where }\approx\text{ is the smallest equivalence relation such that }v\approx v^{-1}\text{ for each }v\in\s I$, and where the number of arrows from $[v]_\approx$ to $[v']_\approx$ equals the cardinality of the set $$\{(v_1,v'_1)\in\s I\times\s I\mid v_1\approx v, v'_1\approx v',v'_1=\s e(\s t(v_1,0))\}.$$ The arrow $[v]_\approx\to[v']_\approx$ contributed by $v'_1=\s e(\s t(v_1,0))$ will be denoted by $a_{(v_1,v'_1)}$.
\end{defn}

\begin{rmk}
For every $[v]_\approx\in {Q_\s M}_0$, there are at most two arrows $\alpha,\beta$ in $Q_{\s M}$ such that $s(\alpha)=s(\beta)=[v]_\approx$; however a similar statement may not be true for arrows with target $[v]_\approx$.
\end{rmk}

The set
\begin{align*}
    \rho_\s M&:=\{p=a_{(v_1,v'_1)}a_{(v_2,v'_2)}\cdots a_{(v_k,v'_k)}\mid p\text{ is a path in }Q, k\geq2\text{ and }\widetilde{\s t}(v_1,0^k)\text{ is undefined}\}
\end{align*}
of paths defines a set of relations on the quiver $Q_{\s M}$.

\begin{thm}\label{prop: string alg from MIA}
Let $\Lambda$ be a string algebra. Then $\C KQ_{{\s M_{\Lambda\delta}}}/\ang{\rho_{{\s M_{\Lambda\delta}}}}\cong\Lambda$.
\end{thm}

\begin{proof}
For each vertex $v\in Q_0$, we have that $1_{(v,1)}$ and $1_{(v,-1)}$ are in $\s I$. Therefore, ${Q_{{\s M_{\Lambda\delta}}}}_0\cong Q_0$. Note that $\alpha\in Q_1$ if and only if $1_{(s(\alpha),-\varsigma(\alpha))}\alpha1_{(t(\alpha),\varepsilon(\alpha))}\in\St$ if and only if the tuple $(1_{(s(\alpha),-\varsigma(\alpha))},0,\alpha)\in{\s t_{\Lambda}}_\delta$ with ${\s e_{\Lambda}}_\delta(\alpha)=1_{(t(\alpha),\varepsilon(\alpha))}$. Hence, ${Q_{{\s M_{\Lambda\delta}}}}_1\cong Q_1$.

For a path $p=a_{(v_1,v'_1)}a_{(v_2,v'_2)}\cdots a_{(v_k,v'_k)}\in\rho$ with $k\geq 2$, where $v_1=1_{(s(p),-\varsigma(p))}$, it follows from the definition of $\s t_{{\s M_{\Lambda\delta}}}$ that $\widetilde{\s t}(v_1,0^k)$ is undefined. Also note that any path $p'$ in $\rho_{{\s M_{\Lambda\delta}}}$ has a left substring $p''$ in $\rho$. Therefore, we have $\ang{\rho}=\ang{\rho_{{\s M_{\Lambda\delta}}}}$.
\end{proof}
\bibliographystyle{alpha}
\bibliography{main}
\end{document}